\documentclass[reqno]{amsart}

\pdfoutput=1
\usepackage{amsrefs}
\usepackage{hyperref,enumerate}
\usepackage{graphicx,multirow}
\usepackage[pdftex]{color}
\usepackage{pdfcolmk}

\newtheorem{theorem}{Theorem}
\newtheorem{lemma}[theorem]{Lemma}

\theoremstyle{definition}

\theoremstyle{remark}

\numberwithin{equation}{section}
\numberwithin{theorem}{section}

\newcommand{\intav}[1]{\mathchoice {\mathop{\vrule width 6pt height 3 pt depth  -2.5pt
\kern -8pt \intop}\nolimits_{\kern -6pt#1}} {\mathop{\vrule width
5pt height 3  pt depth -2.6pt \kern -6pt \intop}\nolimits_{#1}}
{\mathop{\vrule width 5pt height 3 pt depth -2.6pt \kern -6pt
\intop}\nolimits_{#1}} {\mathop{\vrule width 5pt height 3 pt depth
-2.6pt \kern -6pt \intop}\nolimits_{#1}}}

\newcommand{\intavl}[1]{\mathchoice {\mathop{\vrule width 6pt height 3 pt depth  -2.5pt
\kern -8pt \intop}\limits_{\kern -6pt#1}} {\mathop{\vrule width 5pt
height 3  pt depth -2.6pt \kern -6pt \intop}\nolimits_{#1}}
{\mathop{\vrule width 5pt height 3 pt depth -2.6pt \kern -6pt
\intop}\nolimits_{#1}} {\mathop{\vrule width 5pt height 3 pt depth
-2.6pt \kern -6pt \intop}\nolimits_{#1}}}

\newcommand{\mc}{\mathcal}

\newcommand{\R}{\mathbb{R}}

\newcommand{\Z}{\mathbb{Z}}

\newcommand{\T}{\mc{T}}

\newcommand{\p}{{\bf p}}
\renewcommand{\P}[1]{{\mathbb{P}}\left[{#1}\right]}

\begin{document}

\title[Annihilation and coalescence on binary trees]{Annihilation and coalescence on binary trees}

\author[Itai Benjamini]{Itai Benjamini}
\address{Weizmann Institute of Science, Faculty of Mathematics and Computer Science, POB 26, 76100, Rehovot, Israel.}
\email{itai.benjamini@weizmann.ac.il}

\author[Yuri Lima]{Yuri Lima}
\address{Department of Mathematics, University of Maryland, College Park, MD 20742, USA.}
\email{yurilima@gmail.com}


\date{\today}


\begin{abstract}
An infection spreads in a binary tree $\T_n$ of height $n$ as follows: initially, each leaf is either
infected by one of $k$ states or it is not infected at all. The infection state of each leaf is
independently distributed according
to a probability vector $\p=(\p_1,\ldots,\p_{k+1})$. The remaining nodes become
infected or not via annihilation and coalescence: nodes whose two children have the same state
(infected or not) are infected (or not) by this state; nodes whose two children have different states are
not infected; nodes whose only one of the children is infected are infected by this state.
In this note we characterize, for every $\p$, the limiting distribution at the root node of
$\T_n$ as $n$ goes to infinity.

We also consider a variant of the model when $k=2$ and a mutation can happen, with a fixed probability $q$,
at each infection step. We characterize, in terms of $\p$ and $q$, the limiting distribution at the root node of
$\T_n$ as $n$ goes to infinity.

The distribution at the root node is driven by a
dynamical system, and the proofs rely on the analysis of this dynamics.
\end{abstract}



\maketitle

\section{Introduction and statement of results}

Let $\Delta_k$ denote the $k$-dimensional simplex
\begin{align*}
\Delta_k=\left\{\p=(\p_1,\ldots,\p_{k+1})\in\mathbb R^{k+1}:\sum_{i=1}^{k+1}\p_i=1\text{ and }\p_i>0,\forall\,i\right\},
\end{align*}
let $\T_n$ denote the binary tree of height $n$, and fix $\p\in\Delta_k$. The nodes of $\T_n$ are
infected by one of $k$ states $\{1,2,\ldots,k\}$ or not infected as follows.\\

\noindent{\bf Step 1.} Each leaf is infected i.i.d. according to $\p$:
\begin{align}\label{definition of process 1}
\P{\text{leaf is infected by }i}=\p_i\ ,\ \P{\text{leaf is not infected}}=\p_{k+1}.
\end{align}
\noindent{\bf Step 2.} A node adjacent to two leaves is infected or not according to the rules:
\begin{enumerate}[(R1)]
\item[(R1)] if both leaves are not infected, then the node is not infected.
\item[(R2)] if both leaves are infected by the same state, then the node is infected by it,
\item[(R3)] if both leaves are infected by different states, then the node is not infected,
\item[(R4)] if only one of the leaves is infected, then the node is infected by it.\\
\end{enumerate}
\noindent{\bf Step 3.} Repeat Step 2 to each level of $\T_n$.\\

In other words, there is {\bf coalescence} of infection if the states agree and {\bf annihilation}
if they disagree. Let $\p(n)\in\Delta_k$ denote the distribution of the state in the root node of $\T_n$, i.e.
\begin{align*}
\P{\text{root node is infected by } i}=\p_i(n)\ ,\ \P{\text{root node is not infected}}=\p_{k+1}(n).
\end{align*}
In this note, we characterize the limiting behavior of $\p(n)$ as $n$ goes to infinity.

\begin{theorem}\label{main thm 1}
For any $\p\in\Delta_k$, $\p(n)$ converges. Assume that $\p_1\ge\cdots\ge\p_k$.
\begin{enumerate}[(a)]
\item If $\p_1=\cdots=\p_k$, then $\p(n)$ converges to $(\frac{1}{2k-1},\ldots,\frac{1}{2k-1},\frac{k-1}{2k-1})$.
\item If $\p_1=\cdots=\p_i>\p_{i+1}$ for some $i\in\{1,\ldots,k-1\}$, then $\p(n)$ converges to
$(\frac{1}{2i-1},\ldots,\frac{1}{2i-1},0,\ldots,0,\frac{i-1}{2i-1})$, where the entry $\frac{1}{2i-1}$
repeats $i$ times.
\end{enumerate}
\end{theorem}

Thus the asymptotic distribution of $\p(n)$ is uniquely determined by $\p$. This is expected: although
Step 1 is random, all the other steps are deterministic. In particular, if the vector $\p$ is itself random,
and if it is distributed according to a continuous distribution on $\Delta_k$, then almost surely
there will be only one state $i$ that maximizes $\p_i$; only this state will be available for the root
in the limit.

We also analyze a variant of the model when $k=2$. For a fixed $q\in(0,1)$, consider the infection process
with rules (R1), (R2), (R3) and
\begin{enumerate}[(R4)']
\item[(R4)'] if only one of the leaves is infected, then the node is infected by it with probability
$q$ and not infected with probability $1-q$.
\end{enumerate}
Let $\p(n)\in\Delta_2$ denote the distribution of the state in the root node of $\T_n$.

\begin{theorem}\label{main thm 2}
For any $\p\in\Delta_2$, $\p(n)$ converges.
\begin{enumerate}[(a)]
\item If $q>0.5$, then
\begin{align*}
\lim_{n\to\infty}\p(n)=
\left\{\begin{array}{ll}
(1,0,0)&\text{if }\p_1>\p_2,\\
&\\
(0,1,0)&\text{if }\p_1<\p_2,\\
&\\
\left(\frac{2q-1}{4q-1},\frac{2q-1}{4q-1},\frac{1}{4q-1}\right)&\text{if }\p_1=\p_2.\\
\end{array}
\right.
\end{align*}
\item If $q=0.5$, then
\begin{align*}
\lim_{n\to\infty}\p(n)=
\left\{\begin{array}{ll}
(\p_1-\p_2,0,1-\p_1+\p_2)&\text{if }\p_1>\p_2,\\
&\\
(0,-\p_1+\p_2,1+\p_1-\p_2)&\text{if }\p_1<\p_2,\\
&\\
(0,0,1)&\text{if }\p_1=\p_2.\\
\end{array}
\right.
\end{align*}
\item If $q<0.5$, then $\p(n)$ converges to $(0,0,1)$.
\end{enumerate}
\end{theorem}

That is, for large $q$ the behavior is similar to that of Theorem \ref{main thm 1}
(the fixed point inside the simplex varies smoothly with $q$),
there is a phase transition at $q=0.5$, and if $q$ is small then the empty state
dominates. We depict the phase spaces in Figure \ref{picture 1}.

\begin{figure}[hbt!]
\centering
\def\svgwidth{12cm}
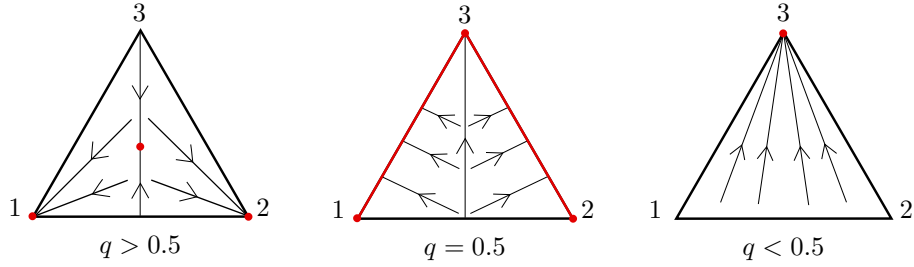
\caption{Phase spaces of $\p(n)$ in Theorem \ref{main thm 2}. The red dots and red lines
represent the possible limit behaviors of $\p(n)$. The black arrows describe the basins of attraction
of each limit behavior, e.g. if $q=0.5$ then $(0,0,1)$ attracts all points in the line
$\p_1=\p_2$.}\label{picture 1}
\end{figure}

The model analyzed here is an instance of a wider setup, in which several
types of particles move in a space and interact as follows: when particles
of the same type meet they coalesce, while when particles of different types meet
they annihilate each other. What is the distribution of the
survivor particles, if any? We collect this and other variants of the
model in Section \ref{section final comments}.

Our setup is completely deterministic: the distribution at the root node
satisfies a quadratic recursive equation, and the analysis of this dynamics
establishes the results. Similar models in trees of larger branching number
lead to recursive equations of higher degree. Not much is known about
these systems.

Recursions appear naturally in the analysis of probabilistic processes on trees.
See e.g. $\S 4.2$ of~\cite{benjamini2000random} and
\cite{dekking1991branching,dekking1991limit,mossel2003noise,pemantle2010critical}.
For a survey on more elaborate recursive equations, see~\cite{aldous2005survey}.

The paper is organized as follows. In Section \ref{section proof of thm 1}
we prove Theorem 1, in Section \ref{section proof of thm 2}
we prove Theorem 2, and in Section \ref{section final comments}
we make final comments and collect further questions.

\section{Proof of Theorem \ref{main thm 1}}\label{section proof of thm 1}

For simplicity of notation, let us assume that the possible states of the nodes on $\T_n$
are $\{1,\ldots,k,k+1\}$: $1,\ldots,k$ represent the infections and $k+1$
represents the empty state (no infection). Given $i,j\in\{1,\ldots,k+1\}$, $i\not=j$, the rules (R1)--(R4)
are depicted in Figure \ref{picture 2}.

\begin{figure}[hbt!]
\centering
\def\svgwidth{12cm}
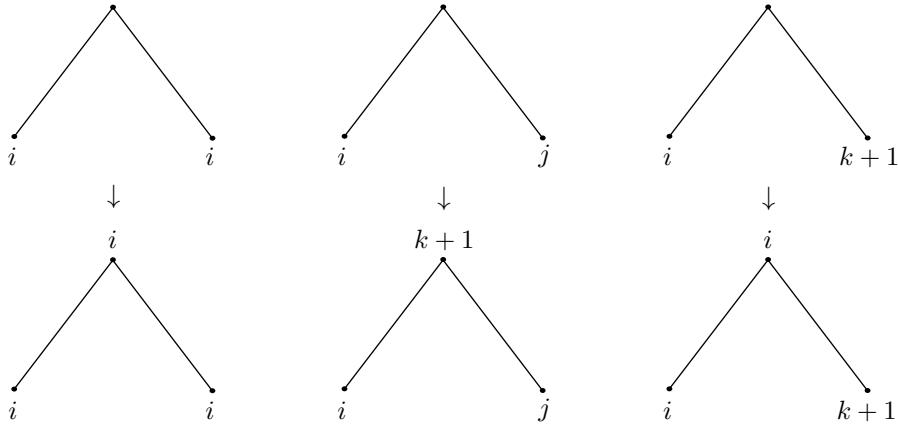
\caption{The rules of the process.}\label{picture 2}
\end{figure}

Let $R_{n+1}$ denote the root node of $\T_{n+1}$, and let $R_{n+1}^1,R_{n+1}^2$ be its two children.
$R_{n+1}^1$ and $R_{n+1}^2$ are root nodes of two independent binary trees of height $n$,
thus their states are independent and distributed according to $\p(n)$. By (R1)--(R4),
the following recursions hold:
\begin{align}\label{main recursion}
\left\{\begin{array}{rcl}
\p_i(n+1)&=&\p_i(n)^2+2\p_i(n)\p_{k+1}(n),\ \ \ i=1,\ldots,k,\text{ and}\\
&&\\
\p_{k+1}(n+1)&=&\p_{k+1}(n)^2+2\displaystyle\sum_{1\le i<j\le k}\p_i(n)\p_j(n).\\
\end{array}
\right.
\end{align}
Define the function $F=(F_1,\ldots,F_k,F_{k+1}):\Delta_k\to\Delta_k$ by
\begin{align*}
\left\{\begin{array}{rcl}
F_i(x_1,\ldots,x_{k+1})&=&x_i^2+2x_ix_{k+1},\ \ \ i=1,\ldots,k,\text{ and}\\
&&\\
F_{k+1}(x_1,\ldots,x_{k+1})&=&x_{k+1}^2+2\displaystyle\sum_{1\le i<j\le k}x_ix_j.\\
\end{array}
\right.
\end{align*}
Thus $\p(n)=F^n(\p)$ for every $n\ge 1$.\\

\noindent (a) Assume that $\p_1=\cdots=\p_k$. Clearly, $\p_1(n)=\cdots=\p_k(n)$
for every $n\ge 1$. So $\{\p_1(n)\}_{n\ge 1}$ satisfies the recursion
\begin{align*}
\p_1(n+1)=\p_1(n)^2+2\p_1(n)\{1-k\p_1(n)\}=\{1-2k\}\p_1(n)^2+2\p_1(n).
\end{align*}
Define $f:(0,k^{-1}]\to\R$ by $f(x)=(1-2k)x^2+2x$. Thus $\p_1(n)=f^n(\p_1)$ for every $n\ge 1$.

\begin{lemma}\label{lemma for quadratic function}
$\overline x=\frac{1}{2k-1}$ is a global attractor of $f$.
\end{lemma}

\begin{proof}
Note that (see Figure \ref{picture 3})
\begin{enumerate}[$\bullet$]
\item $\overline x$ is the unique fixed point of $f$,
\item $f(0,k^{-1}]\subset (0,\overline x]$, and
\item $f|_{(0,\overline x)}:(0,\overline x)\to(0,\overline x)$ is strictly increasing.
\end{enumerate}

\begin{figure}[hbt!]
\centering
\def\svgwidth{4cm}
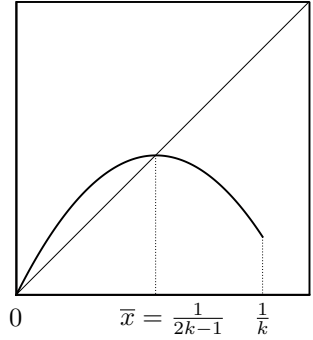
\caption{Graph of $f$.}\label{picture 3}
\end{figure}

For every $x\in (0,k^{-1}]$, the sequence $\{f^n(x)\}_{n\ge 1}$ is increasing and converges to the unique
fixed point $\overline x$ of $f$.
\end{proof}

By Lemma \ref{lemma for quadratic function}, $\p(n)$ converges to
$(\frac{1}{2k-1},\ldots,\frac{1}{2k-1},\frac{k-1}{2k-1})$.\\

\noindent (b) Assume that $\p_1=\cdots=\p_i>\p_{i+1}\ge\cdots\ge\p_k$ for some $i\in\{1,2,\ldots,k-1\}$.
Clearly $\p_1(n)=\cdots=\p_i(n)>\p_{i+1}(n)\ge\cdots\ge\p_k(n)$ for every $n\ge 1$. We will prove that
\begin{align}\label{limit convergence for n colors}
\lim_{n\to\infty}F^n(\p)=
\large\bigg(\underbrace{\dfrac{1}{2i-1},\ldots,\dfrac{1}{2i-1}}_{i},\underbrace{0,\ldots,0}_{k-i},\dfrac{i-1}{2i-1}\large\bigg).
\end{align}
The proof is divided into a few lemmas.

\begin{lemma}\label{lemma E reinforces}
If $x\in\Delta_k$ with $x_1=\displaystyle\max_{1\le j\le k}x_j$, then $F_{k+1}(x)\ge F_2(x)+\cdots+F_k(x)$.
\end{lemma}

\begin{proof}
\begin{eqnarray*}
F_{k+1}(x)-F_2(x)-\cdots-F_k(x)&=&x_{k+1}^2+2\sum_{1\le i<j\le k}x_ix_j-\sum_{i=2}^k(x_i^2+2x_ix_{k+1})\\
&=&x_{k+1}^2-2x_{k+1}\sum_{i=2}^k x_i+2\sum_{1\le i<j\le k}x_ix_j-\sum_{i=2}^k x_i^2\\
&=&\left(x_{k+1}-\sum_{i=2}^k x_i\right)^2-\left(\sum_{i=2}^k x_i\right)^2\\
&&+2\sum_{1\le i<j\le k}x_ix_j-\sum_{i=2}^k x_i^2\\
&=&\left(x_{k+1}-\sum_{i=2}^k x_i\right)^2-\sum_{i=2}^k x_i^2-2\sum_{2\le i<j\le k}x_ix_j\\
&&+2\sum_{1\le i<j\le k}x_ix_j-\sum_{i=2}^k x_i^2\\
&=&\left(x_{k+1}-\sum_{i=2}^k x_i\right)^2+2\sum_{i=2}^k x_1x_i-2\sum_{i=2}^k x_i^2\\
&=&\left(x_{k+1}-\sum_{i=2}^k x_i\right)^2+2\sum_{i=2}^k (x_1-x_i)x_i
\end{eqnarray*}
and every term in the expression above is nonnegative.
\end{proof}

\begin{lemma}\label{lemma auxiliar 1}
$\{\p_1(n)-\p_{i+1}(n)\}_{n\ge 1}$ converges.
\end{lemma}

\begin{proof}
We have
\begin{eqnarray*}
\p_1(n+1)-\p_{i+1}(n+1)&=&\left\{\p_1(n)^2+2\p_1(n)\p_{k+1}(n)\right\}-\\
&&\left\{\p_{i+1}(n)^2+2\p_{i+1}(n)\p_{k+1}(n)\right\}\\
&=&\{\p_1(n)-\p_{i+1}(n)\}\{\p_1(n)+\p_{i+1}(n)+2\p_{k+1}(n)\},
\end{eqnarray*}
that is:
\begin{equation}\label{equation 1}
\p_1(n+1)-\p_{i+1}(n+1)=\{\p_1(n)-\p_{i+1}(n)\}\{\p_1(n)+\p_{i+1}(n)+2\p_{k+1}(n)\}.
\end{equation}
By Lemma \ref{lemma E reinforces},
\begin{eqnarray*}
\p_1(n)+\p_{i+1}(n)+2\p_{k+1}(n)&>&\p_1(n)+2\p_{k+1}(n)\\
&\ge&\p_1(n)+\cdots+\p_{k+1}(n)\\
&=&1,
\end{eqnarray*}
thus $\{\p_1(n)-\p_{i+1}(n)\}_{n\ge 1}$ is a bounded and strictly increasing sequence.
\end{proof}

\begin{lemma}\label{lemma auxiliar 3}
$\{\p_{i+1}(n)\}_{n\ge 1},\ldots,\{\p_k(n)\}_{n\ge 1}$ all converge to zero.
\end{lemma}

\begin{proof}
Because $\p_{i+1}(n)\ge\cdots\ge\p_k(n)$ for every $n\ge 1$, it is enough to prove that
$\{\p_{i+1}(n)\}_{n\ge 1}$ converges to zero.

By equality (\ref{equation 1}), $\{\p_1(n)+\p_{i+1}(n)+2\p_{k+1}(n)\}_{n\ge 1}$
converges to 1: otherwise, $\{\p_1(n)-\p_{i+1}(n)\}_{n\ge 1}$ would be unbounded.
Writing
\begin{align*}
\p_1(n)+\p_{i+1}(n)+2\p_{k+1}(n)=1+\left(\p_{k+1}(n)-\sum_{j=2\atop{j\neq i+1}}^k\p_j(n)\right)=:1+y(n),
\end{align*}
this means that $\{y(n)\}_{n\ge 1}$ converges to zero.
Now write
\begin{align}\label{definition z}
z(n)=\p_{k+1}(n)-\sum_{j=2}^k\p_j(n).
\end{align}
By Lemma \ref{lemma E reinforces}, $y(n)\ge z(n)\ge 0$, so $\{z(n)\}_{n\ge 1}$ also converges to zero.
Consequently, the difference $\{\p_{i+1}(n)=y(n)-z(n)\}_{n\ge 1}$ converges to zero.
\end{proof}

The proofs of Lemmas \ref{lemma auxiliar 1} and \ref{lemma auxiliar 3} imply that
\begin{align}\label{equation 3}
\lim_{n\to\infty}F^n(\p)=\large(\underbrace{x,\ldots,x}_{i},\underbrace{0,\ldots,0}_{k-i},1-ix\large).
\end{align}
By continuity, $x\in (0,i^{-1}]$ is a fixed point of $f=f_i$ as defined in Lemma \ref{lemma for quadratic function},
i.e. $x=\frac{1}{2i-1}$. This concludes the proof of Theorem \ref{main thm 1}.

\section{A variant of the model when $k=2$: proof of Theorem \ref{main thm 2}}\label{section proof of thm 2}

Here, we fix a parameter $q\in(0,1)$ and consider the infection process with rules (R1)--(R3) and (R4)'.
Like in the previous section, we denote the empty state by 3. The rules are depicted in Figure \ref{picture 4}.

\begin{figure}[hbt!]
\centering
\def\svgwidth{12cm}
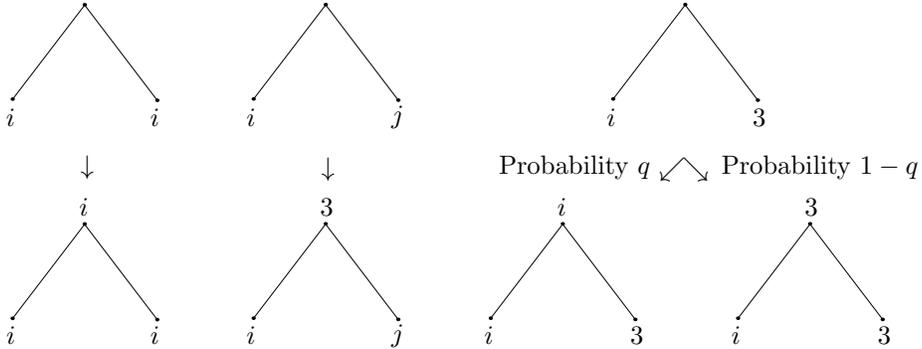
\caption{The rules of the variant process.}\label{picture 4}
\end{figure}

Although in each step the output is random, the distribution of the state in the
root node is again driven by a dynamical system.

Let $\p\in\Delta_2$, and let $\p(n)\in\Delta_2$ denote the distribution of the state in
the root node of $\T_n$. Let $G=(G_1,G_2,G_3):\Delta_2\rightarrow\Delta_2$ be
\begin{align*}
\left\{\begin{array}{rcl}
G_1(x_1,x_2,x_3)&=&x_1^2+2qx_1x_3\\
&&\\
G_2(x_1,x_2,x_3)&=&x_2^2+2qx_2x_3\\
&&\\
G_3(x_1,x_2,x_3)&=&x_3^2+2x_1x_2+2(1-q)x_1x_3+2(1-q)x_2x_3.\\
\end{array}
\right.
\end{align*}
Thus $\p(n)=G^n(\p)$ for every $n\ge 1$.\\

\noindent (a) $q>0.5$: assume first that $\p_1>\p_2$. Thus $\p_1(n)>\p_2(n)$ for every $n\ge 1$.

Firstly, we claim that $\{\p_3(n)\}_{n\ge 1}$ converges to zero. To see this, note that
\begin{eqnarray*}
\p_1(n+1)-\p_2(n+1)&=&\{\p_1(n)-\p_2(n)\}\{\p_1(n)+\p_2(n)+2q\p_3(n)\}\\
&=&\{\p_1(n)-\p_2(n)\}\{1+(2q-1)\p_3(n)\}.
\end{eqnarray*}
Thus $\{\p_1(n)-\p_2(n)\}_{n\ge 1}$ is strictly increasing. Because it is also bounded,
it follows that $\{\p_3(n)\}_{n\ge 1}$ converges to zero.

Now we claim that $\{\p_2(n)\}_{n\ge 1}$ also converges to zero. Just observe that
if $x\in\Delta_2$ with $x_1>x_2$, then $G_3(x)>G_2(x)$:
\begin{align*}
G_3(x)-G_2(x)>(x_3^2+2x_1x_2)-(x_2^2+2x_2x_3)=(x_2-x_3)^2+2(x_1-x_2)x_2>0.
\end{align*}

Thus $\{\p(n)\}_{n\ge 1}$ converges to $(1,0,0)$. Analogously, if $\p_1<\p_2$ then
$\{\p(n)\}_{n\ge 1}$ converges to $(0,1,0)$.

Now assume that $\p_1=\p_2$. In this case, $\{\p_1(n)\}_{n\ge 1}$ satisfies the recursion
\begin{align*}
\p_1(n+1)=\p_1(n)^2+2q\p_1(n)\{1-2\p_1(n)\}
=(1-4q)\p_1(n)^2+2q\p_1(n).
\end{align*}
Define $g:(0,0.5]\to\R$ by $g(x)=(1-4q)x^2+2qx$. Thus $\p_1(n)=g^n(\p_1)$ for every $n\ge 1$.

\begin{figure}[hbt!]
\centering
\def\svgwidth{4cm}
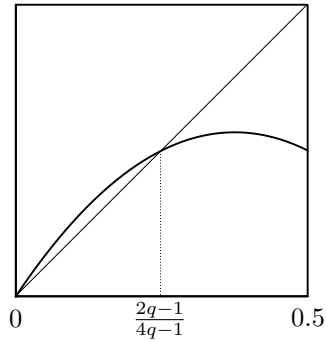
\caption{The graph of $g$.}\label{picture 5}
\end{figure}

Arguing similarly to the proof of Lemma \ref{lemma for quadratic function}, $g$ has a global
attracting fixed point $\frac{2q-1}{4q-1}$ (see Figure \ref{picture 5}), so $\{\p(n)\}_{n\ge 1}$
converges to $\left(\frac{2q-1}{4q-1},\frac{2q-1}{4q-1},\frac{1}{4q-1}\right)$.\\

\noindent (c) $q<0.5$: we have
\begin{align*}
\p_1(n+1)=\p_1(n)\{\p_1(n)+2q\p_3(n)\}<\p_1(n)\{\p_1(n)+\p_3(n)\}<\p_1(n)
\end{align*}
so $\{\p_1(n)\}_{n\ge 1}$ is strictly decreasing. Analogously, $\{\p_2(n)\}_{n\ge 1}$
is strictly decreasing. In particular, these two sequences converge, and so does
$\{\p(n)\}_{n\ge 1}$. By continuity, its limit
$\overline\p=(\overline\p_1,\overline\p_2,\overline\p_3)\in\overline\Delta_2$
satisfies
\begin{align*}
\left\{\begin{array}{rcl}
\overline\p_1&=&\overline\p_1(\overline\p_1+2q\overline\p_3)\\
&&\\
\overline\p_2&=&\overline\p_2(\overline\p_2+2q\overline\p_3).
\end{array}
\right.
\end{align*}
If $\overline\p_1\not=0$, then $\overline\p_1+2q\overline\p_3=1$, so
\begin{align*}
1=\overline\p_1+2q\overline\p_3\le\overline\p_1+\overline\p_3\le\overline\p_1+\overline\p_2+\overline\p_3=1,
\end{align*}
with equality only if $\overline\p_1=1$. This cannot happen, because $\overline\p_1$ is the limit of
a strictly decreasing sequence. Thus $\overline\p_1=0$, and analogously $\overline\p_2=0$, so
$\{\p(n)\}_{n\ge 1}$ converges to $(0,0,1)$.\\

\noindent (b) $q=0.5$: as in (c), $\{\p_1(n)\}_{n\ge 1}$ and $\{\p_2(n)\}_{n\ge 1}$
are strictly decreasing, so $\{\p(n)\}_{n\ge 1}$ converges to some
$\overline\p=(\overline\p_1,\overline\p_2,\overline\p_3)\in\overline\Delta_2$
satisfying
\begin{align*}
\left\{\begin{array}{rcl}
\overline\p_1&=&\overline\p_1(\overline\p_1+\overline\p_3)\\
&&\\
\overline\p_2&=&\overline\p_2(\overline\p_2+\overline\p_3).
\end{array}
\right.
\end{align*}
By the first equality,
\begin{align}\label{equation 2}
\overline\p_1\overline\p_2=\overline\p_1(1-\overline\p_1-\overline\p_3)=0.
\end{align}
Now, note that
\begin{eqnarray*}
\p_1(n+1)-\p_2(n+1)&=&\{\p_1(n)-\p_2(n)\}\{\p_1(n)+\p_2(n)+\p_3(n)\}\\
&=&\p_1(n)-\p_2(n),
\end{eqnarray*}
so $\p_1(n)-\p_2(n)=\p_1-\p_2$ for every $n\ge 1$. Passing to the limit,
$\overline\p_1-\overline\p_2=\p_1-\p_2$. This, together with equality
(\ref{equation 2}), implies that
\begin{align*}
\overline\p=
\left\{\begin{array}{ll}
(\p_1-\p_2,0,1-\p_1+\p_2)&\text{if }\p_1>\p_2,\\
&\\
(0,-\p_1+\p_2,1+\p_1-\p_2)&\text{if }\p_1<\p_2.\\
\end{array}
\right.
\end{align*}
If $\p_1=\p_2$, then
\begin{align*}
\overline\p_1=\overline\p_1(\overline\p_1+\overline\p_3)=\overline\p_1(1-\overline\p_1),
\end{align*}
i.e. $\overline\p_1$ is a fixed point of the map $x\mapsto x(1-x)$. This implies that
$\overline\p=(0,0,1)$, and the proof of Theorem \ref{main thm 2} is complete.

\section{Final comments}\label{section final comments}

\noindent {\bf 1.} As remarked in the introduction, we can consider a wider class of models.
Assume that particles, placed in the vertices of a graph, perform simple
random walks independently, and they annihilate/coalesce when they meet,
depending whether their states are different or not. What is the distribution of the
survivor particles, if any?\\

Using renormalization arguments, the discrete model analyzed here might help the study of related
continuous models. Here is one:
on the circle $\R/\Z$ place $2n+1$ particles uniformly and independently. Each particle moves
on $\R/\Z$ with a random independent
speed, distributed according to the gaussian $\mathcal N(0,1)$.
All particles move simultaneously and when two of them collide they annihilate each other.
What is the speed of the remaining particle? Does it converge to zero as $n$ grows,
or is there a nontrivial limiting distribution?  Another variant is to allow the particles
to perform an independent Brownian motion, each of them with an independent gaussian
diffusion constant.\\

Variants of our model can also be considered. Here we mention three of them.
The first is to change the rules of annihilation/coalescence. Assume there are $k$ possible
states $1,2,\ldots,k$ and we are given a matrix $A=(a_{ij})_{1\le i,j\le k}$ with
$a_{ij}\in\{1,2,\ldots,k\}$. At each step, the particles perform simple random walks
independently and interact according to $A$: when a particle with state $i$ meets a
particle with state $j$, they become a single particle with state $a_{ij}$.\\

The second variant is to consider annihilation/coalescence in a regular
tree of degree $d$, $d\ge 3$, with rules similar to ours: nodes whose all $d$ children have the
same state (infected or not) are infected (or not) by it; nodes with two children with different
states are not infected; nodes whose some children are infected by a single state and the others
are not infected are infected by it.\\

The third variant is to start with other measures on the initial configuration
of the leafs rather than the product measure.\\

\noindent {\bf 2.} Sensitivity of iterated majority with random inputs was studied
in~\cite{mossel2003noise}. It is of interest to consider noise sensitivity and the influence of
leaves subsets in this context, including all the variants, as well.\\

\noindent {\bf 3.} There is no uniform rate of convergence in Theorems \ref{main thm 1}
and \ref{main thm 2}. As an illustration, we prove this for Theorem \ref{main thm 1}.
Let $z(n)$ as in (\ref{definition z}). By Lemma \ref{lemma auxiliar 3}, $z(n)$ converges
to zero. This convergence can be arbitrarily slow: by Lemma \ref{lemma E reinforces},
\begin{align*}
z(n+1)=z(n)^2+2\sum_{i=2}^k(\p_1(n)-\p_i(n))\p_i(n)\ge z(n)^2,
\end{align*}
thus $z(n)\ge z(0)^{2^n}$. In particular, if $z(0)=2^{-2^{-n}}$, then
$z(n)=0.5$.\\

\noindent {\bf 4.} We would like to point out that the quadratic family
appears in the proof of Theorem \ref{main thm 2}, but in a simple way: for each parameter
there is a global attracting fixed point. It would be interesting to describe a set of rules
for the model to force the quadratic family to appear with nontrivial parameters, e.g. parameters
with more periodic points, or even parameters with a horseshoe. See~\cite{brin2002introduction}
for a discussion on the quadratic family.

\section{Acknowledgements}

The authors are thankful to Cyrille Lucas for useful discussions, and to Matheus Secco
for carefully reading the preliminary version of this work.
I.B. is the incumbent of the Renee and Jay Weiss Professorial Chair.
During the preparation of this manuscript, Y.L. was a Postdoctoral Fellow at the
Weizmann Institute of Science, supported by the ERC, grant 239885.
Y.L. is supported by the Brin Fellowship.

\bibliography{binary_tree_bib}

\end{document}